\documentclass{article}
\usepackage{amssymb}
\usepackage{amsfonts}

\newtheorem{theorem}{Theorem}[section]

\newtheorem{corollary}{Corollary}

\newtheorem{definition}[theorem]{Definition}

\newtheorem{lemma}{Lemma}[section]

\newtheorem{proposition}{Proposition}[section]

\newenvironment{proof}[1][Proof]{\noindent\textbf{#1.} }{\ \rule{0.5em}{0.5em}}

\input{tcilatex}

\begin{document}

\title{{\Large AUTOMORPHIC EQUIVALENCE OF LINEAR ALGEBRAS.}}
\author{{\Large A.Tsurkov} \\
%EndAName
Institute of Mathematics and Statistics.\\
University S\~{a}o Paulo. \\
Rua do Mat\~{a}o, 1010 \\
Cidade Universit\'{a}ria \\
S\~{a}o Paulo - SP - Brasil - CEP 05508-090 \\
arkady.tsurkov@gmail.com}
\maketitle

\begin{abstract}
This research is motivated by universal algebraic geometry. We consider in
universal algebraic geometry the some variety of universal algebras $\Theta $
and algebras $H\in \Theta $ from this variety. One of the central question
of the theory is the following: When do two algebras have the same geometry?
What does it mean that the two algebras have the same geometry? The notion
of geometric equivalence of algebras gives a sort of answer to this
question. Algebras $H_{1}$\ and $H_{2}$\ are called geometrically equivalent
if and only if the $H_{1}$-closed sets coincide with the $H_{2}$-closed
sets. The notion of automorphic equivalence is a generalization of the first
notion. Algebras $H_{1}$\ and $H_{2}$\ are called automorphicaly equivalent
if and only if the $H_{1}$-closed sets coincide with the $H_{2}$-closed sets
after some "changing of coordinates".

We can detect the difference between geometric and automorphic equivalence
of algebras of the variety $\Theta $ by researching of the automorphisms of
the category $\Theta ^{0}$ of the finitely generated free algebras of the
variety $\Theta $. By \cite{PlotkinSame} the automorphic equivalence of
algebras provided by inner automorphism degenerated to the geometric
equivalence. So the various differences between geometric and automorphic
equivalence of algebras can be found in the variety $\Theta $ if the factor
group $\mathfrak{A/Y}$ is big. Hear $\mathfrak{A}$ is the group of all
automorphisms of the category\textit{\ }$\Theta ^{0}$, $\mathfrak{Y}$ is a
normal subgroup of all inner automorphisms of the category\textit{\ }$\Theta
^{0}$.

In \cite{PlotkinZhitAutCat} the variety of all Lie algebras and the variety
of all associative algebras over the infinite field $k$ were studied. If the
field $k$ has not nontrivial automorphisms then group $\mathfrak{A/Y}$ in
the first case is trivial and in the second case has order $2$. We consider
in this paper the variety of all linear algebras over the infinite field $k$%
. We prove that group $\mathfrak{A/Y}$ is isomorphic to the group $\left(
U\left( k\mathbf{S}_{\mathbf{2}}\right) \mathfrak{/}U\left( k\left\{
e\right\} \right) \right) \mathfrak{\leftthreetimes }\mathrm{Aut}k$, where $%
\mathbf{S}_{\mathbf{2}}$ is the symmetric group of the set which has $2$
elements, $U\left( k\mathbf{S}_{\mathbf{2}}\right) $ is the group of all
invertible elements of the group algebra $k\mathbf{S}_{\mathbf{2}}$, $e\in 
\mathbf{S}_{\mathbf{2}}$, $U\left( k\left\{ e\right\} \right) $ is a group
of all invertible elements of the subalgebra $k\left\{ e\right\} $, $\mathrm{%
Aut}k$ is the group of all automorphisms of the field $k$.

So even the field $k$ has not nontrivial automorphisms the group $\mathfrak{%
A/Y}$ is infinite. This kind of result is obtained for the first time.

The example of two linear algebras which are automorphically equivalent but
not\textbf{\ }geometrically equivalent is presented in the last section of
this paper. This kind of example is also obtained for the first time.
\end{abstract}

\section{Introduction.\label{intro}}

\setcounter{equation}{0}

In the first two sections we consider some variety $\Theta $ of one-sorted
algebras of the signature $\Omega $. Denote by $X_{0}=\left\{
x_{1},x_{2},\ldots ,x_{n},\ldots \right\} $ a countable set of symbols, and
by $\mathfrak{F}\left( X_{0}\right) $ the set of all finite subsets of $%
X_{0} $. We will consider the category $\Theta ^{0}$, whose objects are all
free algebras $F\left( X\right) $ of the variety $\Theta $ generated by
finite subsets $X\in \mathfrak{F}\left( X_{0}\right) $. Morphisms of the
category $\Theta ^{0}$ are homomorphisms of free algebras.

We denote some time $F\left( X\right) =F\left( x_{1},x_{2},\ldots
,x_{n}\right) $ if $X=\left\{ x_{1},x_{2},\ldots ,x_{n}\right\} $ and even $%
F\left( X\right) =F\left( x\right) $ if $X$ has only one element.

We assume that our variety $\Theta $ possesses the IBN property: for free
algebras $F\left( X\right) ,F\left( Y\right) \in \Theta $ we have $F\left(
X\right) \cong F\left( Y\right) $ if and only if $\left\vert X\right\vert
=\left\vert Y\right\vert $. In this case we have \cite[Theorem 2]%
{PlotkinZhitAutCat} this decomposition%
\begin{equation}
\mathfrak{A=YS}.  \label{decomp}
\end{equation}%
of the group $\mathfrak{A}$ of all automorphisms of the category\textit{\ }$%
\Theta ^{0}$. Hear $\mathfrak{Y}$ is a group of all inner automorphisms of
the category\textit{\ }$\Theta ^{0}$ and $\mathfrak{S}$ is a group of all
strongly stable automorphisms of the category\textit{\ }$\Theta ^{0}$.

\begin{definition}
\label{inner}An automorphism $\Upsilon $ of a category $\mathfrak{K}$ is 
\textbf{inner}, if it is isomorphic as a functor to the identity
automorphism of the category $\mathfrak{K}$.
\end{definition}

This means that for every $A\in \mathrm{Ob}\mathfrak{K}$ there exists an
isomorphism $s_{A}^{\Upsilon }:A\rightarrow \Upsilon \left( A\right) $ such
that for every $\alpha \in \mathrm{Mor}_{\mathfrak{K}}\left( A,B\right) $
the diagram%
\[
\begin{array}{ccc}
A & \overrightarrow{s_{A}^{\Upsilon }} & \Upsilon \left( A\right) \\ 
\downarrow \alpha &  & \Upsilon \left( \alpha \right) \downarrow \\ 
B & \underrightarrow{s_{B}^{\Upsilon }} & \Upsilon \left( B\right)%
\end{array}%
\]%
\noindent commutes.

\begin{definition}
\label{str_stab_aut}\textbf{\hspace{-0.08in}. }\textit{An automorphism $\Phi 
$ of the category }$\Theta ^{0}$\textit{\ is called \textbf{strongly stable}
if it satisfies the conditions:}

\begin{enumerate}
\item[A1)] $\Phi $\textit{\ preserves all objects of }$\Theta ^{0}$\textit{,}

\item[A2)] \textit{there exists a system of bijections }$\left\{ s_{F}^{\Phi
}:F\rightarrow F\mid F\in \mathrm{Ob}\Theta ^{0}\right\} $\textit{\ such
that }$\Phi $\textit{\ acts on the morphisms $\alpha :D\rightarrow F$ of }$%
\Theta ^{0}$\textit{\ by this way: }%
\begin{equation}
\Phi \left( \alpha \right) =s_{F}^{\Phi }\alpha \left( s_{D}^{\Phi }\right)
^{-1},  \label{biject_action}
\end{equation}

\item[A3)] $s_{F}^{\Phi }\mid _{X}=id_{X},$ \textit{\ for every free algebra}
$F=F\left( X\right) $.
\end{enumerate}
\end{definition}

The subgroup $\mathfrak{Y}$ is a normal in $\mathfrak{A}$. We will calculate
the factor group $\mathfrak{A/Y\cong S/S\cap Y}$. This calculation is very
important for universal algebraic geometry.

All definitions of the basic notions of the universal algebraic geometry can
be found, for example, in \cite{PlotkinVarCat}, \cite{PlotkinNotions} and 
\cite{PlotkinSame}. In universal algebraic geometry we consider a "set of
equations" $T\subset F\times F$ in some finitely generated free algebra $F$
of the arbitrary variety of universal algebras $\Theta $ and we "resolve"
these equations in $\mathrm{Hom}\left( F,H\right) $, where $H\in \Theta $.
The set $\mathrm{Hom}\left( F,H\right) $ serves as an "affine space over the
algebra $H$". Denote by $T_{H}^{\prime }$ the set $\left\{ \mu \in \mathrm{%
Hom}\left( F,H\right) \mid T\subset \ker \mu \right\} $. This is the set of
all solutions of the set of equations $T$. For every set of "points" $R$ of
the affine space $\mathrm{Hom}\left( F,H\right) $ we consider a congruence
of equations defined by this set: $R_{H}^{\prime }=\bigcap\limits_{\mu \in
R}\ker \mu $. For every set of equations $T$ we consider its algebraic
closure $T_{H}^{\prime \prime }$ in respect to the algebra $H$. A set $%
T\subset F\times F$ is called $H$-closed if $T=T_{H}^{\prime \prime }$. An $%
H $-closed set is always a congruence.

\begin{definition}
Algebras $H_{1},H_{2}\in \Theta $ are \textbf{geometrically equivalent} if
and only if for every $X\in \mathfrak{F}\left( X_{0}\right) $ and every $%
T\subset F\left( X\right) \times F\left( X\right) $ fulfills $%
T_{H_{1}}^{\prime \prime }=T_{H_{2}}^{\prime \prime }$.
\end{definition}

Denote the family of all $H$-closed congruences in $F$ by $Cl_{H}(F)$. We
can consider the category $C_{\Theta }\left( H\right) $ of the \textit{%
coordinate algebras} connected with the algebra $H\in \Theta $. Objects of
this category are quotient algebras $F\left( X\right) /T$, where $X\in 
\mathfrak{F}\left( X_{0}\right) $, $T\in Cl_{H}(F\left( X\right) )$.
Morphisms of this category are homomorphisms of algebras.

\begin{definition}
\label{automorphic_equivalence}Let $Id\left( H,X\right)
=\bigcap\limits_{\varphi \in \mathrm{Hom}\left( F\left( X\right) ,H\right)
}\ker \varphi $ be the minimal $H$-closed congruence in $\ F\left( X\right) $%
. Algebras $H_{1},H_{2}\in \Theta $ are \textbf{automorphically equivalent}
if and only if there exists a pair $\left( \Phi ,\Psi \right) ,$ where $\Phi
:\Theta ^{0}\rightarrow \Theta ^{0}$ is an automorphism,\textit{\ }$\Psi
:C_{\Theta }\left( H_{1}\right) \rightarrow $\textit{\ }$C_{\Theta }\left(
H_{2}\right) $ is an isomorphism subject to conditions:
\end{definition}

\begin{enumerate}
\item[A.] $\Psi \left( F\left( X\right) /Id\left( H_{1},X\right) \right)
=F\left( Y\right) /Id\left( H_{2},Y\right) $\textit{, where }$\Phi \left(
F\left( X\right) \right) =F\left( Y\right) $\textit{,}

\item[B.] $\Psi \left( F\left( X\right) /T\right) =F\left( Y\right) /%
\widetilde{T}$\textit{, where }$T\in Cl_{H_{1}}(F\left( X\right) )$\textit{, 
}$\widetilde{T}\in Cl_{H_{2}}(F\left( Y\right) )$\textit{,}

\item[C.] $\Psi $\textit{\ takes the natural epimorphism }$\overline{\tau }%
:F\left( X\right) /Id\left( H_{1},X\right) \rightarrow F\left( X\right) /T$%
\textit{\ \ to the natural epimorphism }$\Psi \left( \overline{\tau }\right)
:F\left( Y\right) /Id\left( H_{2},Y\right) \rightarrow F\left( Y\right) /%
\widetilde{T}$\textit{.}
\end{enumerate}

Note that if such a pair $\left( \Phi ,\Psi \right) $ exists, then $\Psi $
is uniquely defined by $\Phi $.

We can say, in certain sense, that automorphic equivalence of algebras is a
coinciding of the structure of closed sets after some "changing of
coordinates" provided by automorphism $\Phi $.

Algebras $H_{1}$\ and $H_{2}$\ are geometrically equivalent if and only if
an inner automorphism $\Phi :\Theta ^{0}\rightarrow \Theta ^{0}$\ provides
the automorphic equivalence of algebras $H_{1}$\ and $H_{2}$. So, only
strongly stable automorphism $\Phi $ can provide us automorphic equivalence
of algebras which not coincides with geometric equivalence of algebras.
Therefore, in some sense, difference from the automorphic equivalence to the
geometric equivalence is measured by the factor group $\mathfrak{A/Y\cong
S/S\cap Y}$.

\section{Verbal operations and strongly stable automorphisms.\label%
{operations}}

\setcounter{equation}{0}

For every word $w=w\left( x_{1},\ldots ,x_{k}\right) \in F\left( X\right) $,
where $F\left( X\right) \in \mathrm{Ob}\Theta ^{0}$, $X=\left\{ x_{1},\ldots
,x_{k}\right\} $ and for every algebra $H\in \Theta $ we can define a $k$%
-ary operation $w_{H}^{\ast }$ on $H$ by 
\[
w_{H}^{\ast }\left( h_{1},\ldots ,h_{k}\right) =w\left( h_{1},\ldots
,h_{k}\right) =\gamma _{h}\left( w\left( x_{1},\ldots ,x_{k}\right) \right)
, 
\]%
where $\gamma _{h}$ is a homomorphism $F\left( X\right) \ni x_{i}\rightarrow
\gamma _{h}\left( x_{i}\right) =h_{i}\in H$, $1\leq i\leq k$. This operation
we call the \textbf{verbal operation} induced on the algebra $H$ by the word 
$w\left( x_{1},\ldots ,x_{k}\right) \in F\left( X\right) $. {A system of
words }$W=\left\{ w_{i}\mid i\in I\right\} $ such that $w_{i}\in F\left(
X_{i}\right) $ {, }$X_{i}=\left\{ x_{1},\ldots ,x_{k_{i}}\right\} ,$ {\
determines a system of }$k_{i}${-ary operations }$\left( w_{i}\right)
_{H}^{\ast }$ on $H$. Denote the set $H$ with the system of these operation
by $H_{W}^{\ast }$.

We have a correspondence between strongly stable automorphisms and {systems
of words which define }the verbal operation and fulfill some conditions.
This correspondence explained in \cite{PlotkinZhitAutCat} and \cite%
{TsurkovAutomEquiv}:\ We denote the signature of our variety $\Theta $ by $%
\Omega $, by $k_{\omega }$ we denote the arity of $\omega $ for every $%
\omega \in \Omega $. We suppose that we have the system of words $W=\left\{
w_{\omega }\mid \omega \in \Omega \right\} $ satisfies the conditions:

\begin{enumerate}
\item[Op1)] $w_{\omega }\left( x_{1},\ldots ,x_{k_{\omega }}\right) \in
F\left( X_{\omega }\right) $, where $X_{\omega }=\left\{ x_{1},\ldots
,x_{k_{\omega }}\right\} $;

\item[Op2)] for every $F=F\left( X\right) \in \mathrm{Ob}\Theta ^{0}$ there
exists an isomorphism $\sigma _{F}:F\rightarrow F_{W}^{\ast }$ such that $%
\sigma _{F}\mid _{X}=id_{X}$.
\end{enumerate}

$F_{W}^{\ast }\in \Theta $ so isomorphisms $\sigma _{F}$ are defined
uniquely by the system of words $W$.

The set $S=\left\{ \sigma _{F}:F\rightarrow F\mid F\in \mathrm{Ob}\Theta
^{0}\right\} $ is a system of bijections which satisfies the conditions:

\begin{enumerate}
\item[B1)] for every homomorphism $\alpha :A\rightarrow B\in \mathrm{Mor}%
\Theta ^{0}$ the mappings $\sigma _{B}\alpha \sigma _{A}^{-1}$ and $\sigma
_{B}^{-1}\alpha \sigma _{A}$ are homomorphisms;

\item[B2)] $\sigma _{F}\mid _{X}=id_{X}$ for every free algebra $F\in 
\mathrm{Ob}\Theta ^{0}$.
\end{enumerate}

So we can define the strongly stable automorphism\textit{\ }by this system
of bijections. This automorphism preserves all objects of $\Theta ^{0}$ and
acts on morphism of $\Theta ^{0}$ by formula (\ref{biject_action}), where $%
s_{F}^{\Phi }=$ $\sigma _{F}$.

Vice versa if we have a strongly stable automorphism $\Phi $ of the category 
$\Theta ^{0}$ then its system of bijections $S=\left\{ s_{F}^{\Phi
}:F\rightarrow F\mid F\in \mathrm{Ob}\Theta ^{0}\right\} $ defined uniquely.
Really, if $F\in \mathrm{Ob}\Theta ^{0}$ and $f\in F$ then%
\begin{equation}
s_{F}^{\Phi }\left( f\right) =s_{F}^{\Phi }\alpha \left( x\right) =\left(
s_{F}^{\Phi }\alpha \left( s_{D}^{\Phi }\right) ^{-1}\right) \left( x\right)
=\left( \Phi \left( \alpha \right) \right) \left( x\right) ,
\label{autom_bijections}
\end{equation}%
where $D=F\left( x\right) $ - $1$-generated free linear algebra - and $%
\alpha :D\rightarrow F$ homomorphism such that $\alpha \left( x\right) =f$.
Obviously that this system of bijections $S=\left\{ s_{F}^{\Phi
}:F\rightarrow F\mid F\in \mathrm{Ob}\Theta ^{0}\right\} $ fulfills
conditions B1 and B2 with $\sigma _{F}=s_{F}^{\Phi }$.

If we have a system of bijections $S=\left\{ \sigma _{F}:F\rightarrow F\mid
F\in \mathrm{Ob}\Theta ^{0}\right\} $ which fulfills conditions B1 and B2
than we can define the system of words $W=\left\{ w_{\omega }\mid \omega \in
\Omega \right\} $ satisfies the conditions Op1 and Op2 by formula%
\begin{equation}
w_{\omega }\left( x_{1},\ldots ,x_{k_{\omega }}\right) =\sigma _{F_{\omega
}}\left( \omega \left( \left( x_{1},\ldots ,x_{k_{\omega }}\right) \right)
\right) \in F_{\omega },  \label{der_veb_opr}
\end{equation}%
where $F_{\omega }=F\left( X_{\omega }\right) $.

By formulas (\ref{autom_bijections}) and (\ref{der_veb_opr}) we can check
that there are

\begin{enumerate}
\item one to one and onto correspondence between strongly stable
automorphisms of the category $\Theta ^{0}$ and systems of bijections
satisfied the conditions B1 and B2

\item one to one and onto correspondence between systems of bijections
satisfied the conditions B1 and B2 and systems of words satisfied the
conditions Op1 and Op2.
\end{enumerate}

So we can find a strongly stable automorphism $\Phi $ of the category $%
\Theta ^{0}$ by finding a system of words which fulfills conditions Op1 and
Op2.

\section{Verbal operations in linear algebras.\label%
{operations_in_linear_alg}}

\setcounter{equation}{0}

From now on, we consider the variety $\Theta $ of all linear algebras over
infinite field $k$. We consider linear algebras as one-sorted universal
algebras, i. e., multiplication by scalar we consider as $1$-ary operation
for every $\lambda \in k$: $H\ni h\rightarrow \lambda h\in H$ where $H\in
\Theta $. Hence the signature $\Omega $ of algebras of our variety contains
these operations: $0$-ary operation $0$; $\left\vert k\right\vert $ $1$-ary
operations of multiplications by scalars; $1$-ary operation $-:h\rightarrow
-h$, where $h\in H$, $H\in \Theta $; $2$-ary operation $\cdot $ and $2$-ary
operation $+$. We will finding the system of words $W=\left\{ w_{\omega
}\mid \omega \in \Omega \right\} $ satisfies the conditions Op1 and Op2. We
denote the words corresponding to these operations by $w_{0}$, $w_{\lambda }$
for all $\lambda \in k$, $w_{-}$, $w_{\cdot }$, $w_{+}$.

For arbitrary $F\left( X\right) \in \mathrm{Ob}\Theta ^{0}$ we denote $%
F\left( X\right) =\bigoplus\limits_{i=1}^{\infty }F_{i}$ the decomposition
to the linear spaces of elements which are homogeneous according the sum of
degrees of generators from the set $X$. We also denote the two-sides ideals $%
\bigoplus\limits_{i=j}^{\infty }F_{i}=F^{j}$. From now on, the word "ideal"
means two sided ideal of linear algebra.

We denote the group of all automorphisms of the field $k$ by $\mathrm{Aut}k$.

Our variety $\Theta $ possesses the IBN property, because $\left\vert
X\right\vert =\dim F/F^{2}$ fulfills for all free algebras $F=F\left(
X\right) \in \Theta $. So we have the decomposition (\ref{decomp}) for group
of all automorphisms of the category $\Theta ^{0}$.

Now we need to prove one technical fact about $1$-generated free linear
algebra $F\left( x\right) $.

\begin{lemma}
\label{monomials}Let $\left\{ u_{1},\ldots ,u_{r}\right\} $ is the set of
all monomials of degree $n$ in $F\left( x\right) $ (basis of $F_{n}$), $%
\left\{ v_{1},\ldots ,v_{t}\right\} $ is the set of all monomials of degree $%
m$ in $F\left( x\right) $ (basis of $F_{m}$), $\varphi $ is an arbitrary
function from $\left\{ 1,\ldots ,n\right\} $ to $\left\{ 1,\ldots ,t\right\} 
$. Denote by $\varphi \left( u_{l}\right) $ the monomial which is a results
of substitution into monomial $u_{l}$ ($1\leq l\leq r$) instead $j$-th from
left entry of $x$ the monomial $v_{\varphi \left( j\right) }$ ($1\leq j\leq
n $). All these monomials are distinct, i. e., $\varphi _{1}\left(
u_{l_{1}}\right) =\varphi _{2}\left( u_{l_{2}}\right) $ if and only if $%
\varphi _{1}=\varphi _{2}$ and $u_{l_{1}}=u_{l_{2}}$, where $\varphi
_{1},\varphi _{2}:\left\{ 1,\ldots ,n\right\} \rightarrow \left\{ 1,\ldots
,t\right\} $, $u_{l_{1}},u_{l_{2}}\in \left\{ u_{1},\ldots ,u_{r}\right\} $.
\end{lemma}

\begin{proof}
We will prove this lemma by induction by $n$ - degree of monomials from $%
\left\{ u_{1},\ldots ,u_{r}\right\} $. The claim of the lemma is trivial for 
$n=1$. We assume that the claim of the lemma is proved for monomials which
have degree $<n$. We suppose that $\varphi _{1}\left( u_{l_{1}}\right)
=\varphi _{2}\left( u_{l_{2}}\right) $, where $\deg u_{l_{1}}=\deg
u_{l_{2}}=n>1$, $\varphi _{1},\varphi _{2}:\left\{ 1,\ldots ,n\right\}
\rightarrow \left\{ 1,\ldots ,t\right\} $. $u_{l_{i}}=u_{l_{i}}^{(1)}\cdot
u_{l_{i}}^{(2)}$, where $i=1,2$. We denote $\deg u_{l_{i}}^{(1)}=c_{i}$. $%
1\leq c_{i}<n$ for $i=1,2$. For $i=1,2$ we have $\varphi _{i}\left(
u_{l_{i}}\right) =\varphi _{i}^{(1)}\left( u_{l_{i}}^{(1)}\right) \cdot
\varphi _{i}^{(2)}\left( u_{l_{i}}^{(2)}\right) $, where $\varphi
_{i}^{(1)}:\left\{ 1,\ldots ,c_{i}\right\} \rightarrow \left\{ 1,\ldots
,t\right\} $, $\varphi _{i}^{(2)}:\left\{ 1,\ldots ,n-c_{i}\right\}
\rightarrow \left\{ 1,\ldots ,t\right\} $, $\varphi _{i}^{(1)}\left(
j\right) =\varphi _{i}\left( j\right) $ for $1\leq j\leq c_{i}$, $\varphi
_{i}^{(2)}\left( j\right) =\varphi _{i}\left( c_{i}+j\right) $ for $1\leq
j\leq n-c_{1}$. $\varphi _{1}\left( u_{l_{1}}\right) =\varphi _{2}\left(
u_{l_{2}}\right) $ if and only if $\varphi _{1}^{(1)}\left(
u_{l_{1}}^{(1)}\right) =\varphi _{2}^{(1)}\left( u_{l_{2}}^{(1)}\right) $
and $\varphi _{1}^{(2)}\left( u_{l_{1}}^{(2)}\right) =\varphi
_{2}^{(2)}\left( u_{l_{2}}^{(2)}\right) $. If $c_{1}\neq c_{2}$ then $\deg
\varphi _{1}^{(1)}\left( u_{l_{1}}^{(1)}\right) =c_{1}m\neq \deg \varphi
_{2}^{(1)}\left( u_{l_{2}}^{(1)}\right) =c_{2}m$, hence $\varphi
_{1}^{(1)}\left( u_{l_{1}}^{(1)}\right) \neq \varphi _{2}^{(1)}\left(
u_{l_{2}}^{(1)}\right) $ and $\varphi _{1}\left( u_{l_{1}}\right) \neq
\varphi _{2}\left( u_{l_{2}}\right) $. So $c_{1}=c_{2}$ and, by our
assumption, $\varphi _{1}^{(1)}=\varphi _{2}^{(1)}$, $%
u_{l_{1}}^{(1)}=u_{l_{2}}^{(1)}$, $\varphi _{1}^{(2)}=\varphi _{2}^{(2)}$, $%
u_{l_{1}}^{(2)}=u_{l_{2}}^{(2)}$. Therefore $\varphi _{1}=\varphi _{2}$ and $%
u_{l_{1}}=u_{l_{2}}$.
\end{proof}

\begin{corollary}
\label{substitute}Let $f\left( x\right) ,g\left( x\right) \in F\left(
X\right) $. $f\left( g\left( x\right) \right) $ is a result of substitution
of $g\left( x\right) $ in $f\left( x\right) $ instead $x$. $f\left( g\left(
x\right) \right) \in F_{1}$ if and only if $f\left( x\right) ,g\left(
x\right) \in F_{1}$.
\end{corollary}

\begin{proof}
We write $f\left( x\right) $ and $g\left( x\right) $ as sum of its
homogeneous components: $f\left( x\right) =f_{1}\left( x\right) +f_{2}\left(
x\right) +\ldots +f_{n}\left( x\right) $, $g\left( x\right) =g_{1}\left(
x\right) +g_{2}\left( x\right) +\ldots +g_{m}\left( x\right) $, $f_{i}\left(
x\right) ,g_{i}\left( x\right) \in F_{i}$. We assume that $n>1$ or $m>1$, $%
f_{n}\left( x\right) \neq 0$ and $g_{m}\left( x\right) \neq 0$. $f\left(
g\left( x\right) \right) =f_{1}\left( g\left( x\right) \right) +f_{2}\left(
g\left( x\right) \right) +\ldots +f_{n}\left( g\left( x\right) \right) $.
Addenda of the maximal possible degree of $x$, which can appear in $f\left(
g\left( x\right) \right) $, i. e., addenda of degree $nm$ can appear in $%
f_{n}\left( g\left( x\right) \right) $. They coincide with addenda of $%
f_{n}\left( g_{m}\left( x\right) \right) $. Denote $f_{n}\left( x\right)
=\lambda _{1}u_{1}+\ldots +\lambda _{r}u_{r}$, $g_{m}\left( x\right) =\mu
_{1}v_{1}+\ldots +\mu _{t}v_{t}$, where $\left\{ u_{1},\ldots ,u_{r}\right\} 
$ is the set of all monomials of degree $n$ in $F\left( x\right) $, $\left\{
v_{1},\ldots ,v_{t}\right\} $ is the set of all monomials of degree $m$ in $%
F\left( x\right) $, $\lambda _{i},\mu _{j}\in k$. Not all $\left\{ \lambda
_{1},\ldots ,\lambda _{r}\right\} $ and not all $\left\{ \mu _{1},\ldots
,\mu _{t}\right\} $ are equal to $0$ by our assumption. $f_{n}\left(
g_{m}\left( x\right) \right) =\lambda _{1}u_{1}\left( g_{m}\left( x\right)
\right) +\ldots +\lambda _{r}u_{r}\left( g_{m}\left( x\right) \right) $. If
we open the brackets in $u_{l}\left( g_{m}\left( x\right) \right)
=u_{l}\left( \mu _{1}v_{1}+\ldots +\mu _{t}v_{t}\right) $ ($1\leq l\leq r$),
we obtain addenda, which are results of substitution into monomial $u_{l}$
instead all entry of $x$ some monomial from $\mu _{1}v_{1},\ldots ,\mu
_{t}v_{t}$ in all possible options. We can say more formal: for every
function $\varphi :\left\{ 1,\ldots ,n\right\} \rightarrow \left\{ 1,\ldots
,t\right\} $ we obtain an addendum which is a results of substitution into
monomial $u_{l}$ instead $j$-th from left entry of $x$ the monomial $\mu
_{\varphi \left( j\right) }v_{\varphi \left( j\right) }$ ($1\leq j\leq n$).
Therefore all addenda, which we obtain after the opening of the brackets in $%
f_{n}\left( g_{m}\left( x\right) \right) $, distinct from the monomials
discussed in \textbf{Lemma \ref{monomials}} only by coefficients. All these
addenda have degree $nm>1$, because $n>1$ or $m>1$. So addenda of $%
f_{n}\left( g_{m}\left( x\right) \right) $ can not cancel one another by 
\textbf{Lemma \ref{monomials}}. These addenda can not be canceled by other
addenda $f\left( g\left( x\right) \right) $, because all other addenda have
degree $<nm$. Therefore all these addenda equal to $0$, because $f\left(
g\left( x\right) \right) \in F_{1}$. For $l\in \left\{ 1,\ldots ,r\right\} $
and $j\in \left\{ 1,\ldots ,t\right\} $ we take the addendum which is a
results of substitution into monomial $\lambda _{l}u_{l}$ instead all
entries of $x$ the monomial $\mu _{j}v_{j}$. The coefficient of this
addendum is $\lambda _{l}\mu _{j}^{n}=0$. So $\lambda _{l}\mu _{j}=0$ for
all $l\in \left\{ 1,\ldots ,r\right\} $ and all $j\in \left\{ 1,\ldots
,t\right\} $. It contradicts the fact that $f_{n}\left( x\right) \neq 0$ and 
$g_{m}\left( x\right) \neq 0$.
\end{proof}

\setcounter{corollary}{0}

\begin{theorem}
\label{words}The system of words%
\begin{equation}
W=\left\{ w_{0},w_{\lambda }\left( \lambda \in k\right)
,w_{-},w_{+},w_{\cdot }\right\}  \label{words_list}
\end{equation}%
satisfies the conditions Op1 and Op2 if and only if $w_{0}=0$, $w_{\lambda
}=\varphi \left( \lambda \right) x_{1}$, $w_{-}=-x_{1}$, $w_{+}=x_{1}+x_{2}$%
, $w_{\cdot }=ax_{1}x_{2}+bx_{2}x_{1}$, where $\varphi $ is an automorphism
of the field $k$, $a,b\in k$, $a\neq \pm b$.
\end{theorem}

\begin{proof}
Let $W$ (see (\ref{words_list}) ) satisfies the conditions Op1 and Op2.

$w_{0}$ is an element of the $0$-generated free linear algebra. There is
only one element in this algebra: $0$. This is the only one opportunity for $%
w_{0}$.

$w_{\lambda }\in F\left( x\right) $ for every $\lambda \in k$. Denote
multiplications by scalars in $\left( F\left( x\right) \right) _{W}^{\ast }$
by $\ast $, i. e., $\lambda \ast f=w_{\lambda }\left( f\right) $ for every $%
f\in F\left( x\right) $ and every $\lambda \in k$. $\left( F\left( x\right)
\right) _{W}^{\ast }\in \Theta $, therefore, if $\lambda =0$ then $0\ast
x=w_{0}\left( x\right) =0$. If $\lambda \neq 0$ then%
\[
1\ast x=\left( \lambda ^{-1}\lambda \right) \ast x=\lambda ^{-1}\ast \left(
\lambda \ast x\right) =w_{\lambda ^{-1}}\left( w_{\lambda }\left( x\right)
\right) =x. 
\]%
Hence $w_{\lambda }=\varphi \left( \lambda \right) x$ by \textbf{Corollary 1}
from \textbf{Lemma \ref{monomials}}, where $\varphi \left( \lambda \right)
\in k$. We can write $\varphi \left( 0\right) =0$. Also we have that for all 
$\lambda _{1},\lambda _{2}\in k$ fulfills%
\[
\left( \lambda _{1}\lambda _{2}\right) \ast x=\varphi \left( \lambda
_{1}\lambda _{2}\right) x 
\]
and 
\[
\left( \lambda _{1}\lambda _{2}\right) \ast x=\lambda _{1}\ast \left(
\lambda _{2}\ast x\right) =\lambda _{1}\ast \left( \varphi \left( \lambda
_{2}\right) x\right) = 
\]%
\[
\varphi \left( \lambda _{1}\right) \left( \varphi \left( \lambda _{2}\right)
x\right) =\left( \varphi \left( \lambda _{1}\right) \varphi \left( \lambda
_{2}\right) \right) x. 
\]%
So $\varphi \left( \lambda _{1}\right) \varphi \left( \lambda _{2}\right)
=\varphi \left( \lambda _{1}\lambda _{2}\right) $. If $\mu \in k\setminus
\left\{ 0\right\} $, then the $1$-ary operation of multiplication by scalar $%
\mu $ is a verbal operation defined by some word $w_{\mu }^{\ast }\left(
x\right) \in \left( F\left( x\right) \right) _{W}^{\ast }$, written be the
operations defined by system of words $W$ - see \cite[Proposition 4.2]%
{TsurkovAutomEquiv}. Hence, $\mu f=w_{\mu }^{\ast }\left( f\right) $ holds
for every $f\in F\left( x\right) $. Also there is $w_{\mu ^{-1}}^{\ast
}\left( x\right) \in \left( F\left( x\right) \right) _{W}^{\ast }$ such that 
$\mu ^{-1}f=w_{\mu ^{-1}}^{\ast }\left( f\right) $ for every $f\in F\left(
x\right) $. $x=\mu ^{-1}\left( \mu x\right) =w_{\mu ^{-1}}^{\ast }\left(
w_{\mu }^{\ast }\left( x\right) \right) $. There exists by Op2 an
isomorphism $\sigma _{F\left( x\right) }:F\left( x\right) \rightarrow \left(
F\left( x\right) \right) _{W}^{\ast }$ such that $\sigma _{F\left( x\right)
}\left( x\right) =x$. So $\left( F\left( x\right) \right) _{W}^{\ast }$ is
also $1$-generated free linear algebra of $\Theta $ with the free generator $%
x$. Hence there exists a decomposition $\left( F\left( x\right) \right)
_{W}^{\ast }=\bigoplus\limits_{i=1}^{\infty }F_{i}^{\ast }$, where $%
F_{i}^{\ast }$ are linear spaces of elements which are homogeneous according
the degree of $x$ but in respect of operations defined by system of words $W$%
. Therefore $w_{\mu }^{\ast }\left( x\right) =\lambda \ast x$, where $%
\lambda \in k$, by \textbf{Corollary 1 }from \ \textbf{Lemma \ref{monomials}}%
. So $\mu x=\lambda \ast x=\varphi \left( \lambda \right) x$ and $\mu
=\varphi \left( \lambda \right) $, hence $\varphi :k\rightarrow k$ is a
surjection.

$w_{+}\in F\left( x_{1},x_{2}\right) =F$. There exists $n\in 
%TCIMACRO{\U{2115} }%
%BeginExpansion
\mathbb{N}
%EndExpansion
$, such that%
\[
w_{+}\left( x_{1},x_{2}\right) =p_{1}\left( x_{1},x_{2}\right) +p_{2}\left(
x_{1},x_{2}\right) +\ldots +p_{n}\left( x_{1},x_{2}\right) , 
\]%
where $p_{i}\left( x_{1},x_{2}\right) \in F_{i}$, $1\leq i\leq n$. We have
for every $\lambda \in k$ that 
\[
w_{+}\left( \lambda \ast x_{1},\lambda \ast x_{2}\right) =\lambda \ast
w_{+}\left( x_{1},x_{2}\right) =\varphi \left( \lambda \right) w_{+}\left(
x_{1},x_{2}\right) = 
\]%
\[
\varphi \left( \lambda \right) p_{1}\left( x_{1},x_{2}\right) +\varphi
\left( \lambda \right) p_{2}\left( x_{1},x_{2}\right) +\ldots +\varphi
\left( \lambda \right) p_{n}\left( x_{1},x_{2}\right) 
\]%
and%
\[
w_{+}\left( \lambda \ast x_{1},\lambda \ast x_{2}\right) =p_{1}\left(
\lambda \ast x_{1},\lambda \ast x_{2}\right) +p_{2}\left( \lambda \ast
x_{1},\lambda \ast x_{2}\right) +\ldots +p_{n}\left( \lambda \ast
x_{1},\lambda \ast x_{2}\right) = 
\]%
\[
p_{1}\left( \varphi \left( \lambda \right) x_{1},\varphi \left( \lambda
\right) x_{2}\right) +p_{2}\left( \varphi \left( \lambda \right)
x_{1},\varphi \left( \lambda \right) x_{2}\right) +\ldots +p_{n}\left(
\varphi \left( \lambda \right) x_{1},\varphi \left( \lambda \right)
x_{2}\right) = 
\]%
\[
\varphi \left( \lambda \right) p_{1}\left( x_{1},x_{2}\right) +\left(
\varphi \left( \lambda \right) \right) ^{2}p_{2}\left( x_{1},x_{2}\right)
+\ldots +\left( \varphi \left( \lambda \right) \right) ^{n}p_{n}\left(
x_{1},x_{2}\right) . 
\]%
We can take $\lambda \in k$ such that $\varphi \left( \lambda \right) $ is
not a solution of any equation $x^{i}=x$, where $2\leq i\leq n$. So, $%
p_{i}\left( x_{1},x_{2}\right) =0$ for $2\leq i\leq n$ by equality of the
homogeneous components. Therefore $w_{+}=\alpha x_{1}+\beta x_{2}$, where $%
\alpha ,\beta \in k$. If we denote the operation defined by $w_{+}$ in $%
\left( F\left( x_{1},x_{2}\right) \right) _{W}^{\ast }$ by $\bot $, then $%
x_{1}\bot x_{2}=x_{2}\bot x_{1}$ holds, so $\alpha x_{1}+\beta x_{2}=\alpha
x_{2}+\beta x_{1}$ and $\alpha =\beta $. Also $x_{1}\bot 0=x_{1}$ holds and $%
\alpha x_{1}=x_{1}$, so $\alpha =\beta =1$. Now, by consideration of $%
F\left( x\right) $, we can conclude that for all $\lambda _{1},\lambda
_{2}\in k$ fulfills 
\[
\varphi \left( \lambda _{1}+\lambda _{2}\right) x=\left( \lambda
_{1}+\lambda _{2}\right) \ast x=\lambda _{1}\ast x\bot \lambda _{2}\ast x= 
\]%
\[
\lambda _{1}\ast x+\lambda _{2}\ast x=\varphi \left( \lambda _{1}\right)
x+\varphi \left( \lambda _{2}\right) x=\left( \varphi \left( \lambda
_{1}\right) +\varphi \left( \lambda _{2}\right) \right) x, 
\]%
so $\varphi \left( \lambda _{1}+\lambda _{2}\right) =$ $\varphi \left(
\lambda _{1}\right) +\varphi \left( \lambda _{2}\right) $ and $\varphi $ is
an automorphism of the field $k$.

Its clear now that $w_{-}=-x\in F\left( x\right) $, because%
\[
w_{-}\left( x\right) =-1\ast x=\varphi \left( -1\right) x=\left( -1\right)
x=-x. 
\]

$w_{\cdot }\in F\left( x_{1},x_{2}\right) $. We write $w_{\cdot }$ as sum of
its homogeneous components according the degree of $x_{1}$:%
\[
w_{\cdot }\left( x_{1},x_{2}\right) =p_{0}\left( x_{1},x_{2}\right)
+p_{1}\left( x_{1},x_{2}\right) +p_{2}\left( x_{1},x_{2}\right) +\ldots
+p_{n}\left( x_{1},x_{2}\right) . 
\]%
We denote the operation defined by $w_{\cdot }$ in $\left( F\left(
x_{1},x_{2}\right) \right) _{W}^{\ast }$ by $\times $. So we have for every $%
\lambda \in k$ that%
\[
\left( \lambda \ast x_{1}\right) \times x_{2}=\lambda \ast \left(
x_{1}\times x_{2}\right) =\varphi \left( \lambda \right) w_{\cdot }\left(
x_{1},x_{2}\right) = 
\]%
\[
\varphi \left( \lambda \right) p_{0}\left( x_{1},x_{2}\right) +\varphi
\left( \lambda \right) p_{1}\left( x_{1},x_{2}\right) +\varphi \left(
\lambda \right) p_{2}\left( x_{1},x_{2}\right) +\ldots +\varphi \left(
\lambda \right) p_{n}\left( x_{1},x_{2}\right) . 
\]%
and%
\[
\left( \lambda \ast x_{1}\right) \times x_{2}=w_{\cdot }\left( \varphi
\left( \lambda \right) x_{1},x_{2}\right) = 
\]%
\[
p_{0}\left( \varphi \left( \lambda \right) x_{1},x_{2}\right) +p_{1}\left(
\varphi \left( \lambda \right) x_{1},x_{2}\right) +p_{2}\left( \varphi
\left( \lambda \right) x_{1},x_{2}\right) +\ldots +p_{n}\left( \varphi
\left( \lambda \right) x_{1},x_{2}\right) = 
\]%
\[
p_{0}\left( x_{1},x_{2}\right) +\varphi \left( \lambda \right) p_{1}\left(
x_{1},x_{2}\right) +\left( \varphi \left( \lambda \right) \right)
^{2}p_{2}\left( x_{1},x_{2}\right) +\ldots +\left( \varphi \left( \lambda
\right) \right) ^{n}p_{n}\left( x_{1},x_{2}\right) . 
\]%
We can take, as above, $\lambda \in k$ such that by equality of the
homogeneous components we obtain that $w_{\cdot }\left( x_{1},x_{2}\right)
=p_{1}\left( x_{1},x_{2}\right) $. Now we write $w_{\cdot }\left(
x_{1},x_{2}\right) =p_{1}\left( x_{1},x_{2}\right) $ as sum of its
homogeneous components according the degree of $x_{2}$:%
\[
w_{\cdot }\left( x_{1},x_{2}\right) =r_{0}\left( x_{1},x_{2}\right)
+r_{1}\left( x_{1},x_{2}\right) +r_{2}\left( x_{1},x_{2}\right) +\ldots
+r_{m}\left( x_{1},x_{2}\right) . 
\]%
We have for every $\lambda \in k$ that%
\[
x_{1}\times \left( \lambda \ast x_{2}\right) =\lambda \ast \left(
x_{1}\times x_{2}\right) =\varphi \left( \lambda \right) w_{\cdot }\left(
x_{1},x_{2}\right) = 
\]%
\[
\varphi \left( \lambda \right) r_{0}\left( x_{1},x_{2}\right) +\varphi
\left( \lambda \right) r_{1}\left( x_{1},x_{2}\right) +\varphi \left(
\lambda \right) r_{2}\left( x_{1},x_{2}\right) +\ldots +\varphi \left(
\lambda \right) r_{m}\left( x_{1},x_{2}\right) . 
\]%
and%
\[
x_{1}\times \left( \lambda \ast x_{2}\right) =w_{\cdot }\left( x_{1},\varphi
\left( \lambda \right) x_{2}\right) = 
\]%
\[
r_{0}\left( x_{1},\varphi \left( \lambda \right) x_{2}\right) +r_{1}\left(
x_{1},\varphi \left( \lambda \right) x_{2}\right) +r_{2}\left( x_{1},\varphi
\left( \lambda \right) x_{2}\right) +\ldots +r_{m}\left( x_{1},\varphi
\left( \lambda \right) x_{2}\right) = 
\]%
\[
r_{0}\left( x_{1},x_{2}\right) +\varphi \left( \lambda \right) r_{1}\left(
x_{1},x_{2}\right) +\varphi \left( \lambda \right) ^{2}r_{2}\left(
x_{1},x_{2}\right) +\ldots +\varphi \left( \lambda \right) ^{m}r_{m}\left(
x_{1},x_{2}\right) . 
\]%
And, as above, we can conclude that $w_{\cdot }\left( x_{1},x_{2}\right)
=r_{1}\left( x_{1},x_{2}\right) $ where $r_{1}\left( x_{1},x_{2}\right) $ is
a homogeneous element of $F\left( x_{1},x_{2}\right) $ such that $\deg
_{x_{1}}r_{1}\left( x_{1},x_{2}\right) =1$ and $\deg _{x_{2}}r_{1}\left(
x_{1},x_{2}\right) =1$. Therefore $w_{\cdot }\left( x_{1},x_{2}\right)
=ax_{1}x_{2}+bx_{2}x_{1}$, where $a,b\in k$. If $a=b$ then the operation
defined by $w_{\cdot }\left( x_{1},x_{2}\right) $ is commutative. If\ $a=-b$
then the operation defined by $w_{\cdot }\left( x_{1},x_{2}\right) $ is
anticommutative. The isomorphisms $\sigma _{F}:F\rightarrow F_{W}^{\ast }$,
where $F\in \mathrm{Ob}\Theta ^{0}$ can not exists in both these cases if $F$
is not a $0$-generated free algebra.

Therefore we prove that if the system of words (\ref{words_list}) satisfies
the conditions Op1 and Op2 then $w_{0}=0$, $w_{\lambda }=\varphi \left(
\lambda \right) x_{1}$ for all $\lambda \in k$, $w_{-}=-x_{1}$, $%
w_{+}=x_{1}+x_{2}$, $w_{\cdot }=ax_{1}x_{2}+bx_{2}x_{1}$, where $\varphi $
is an automorphism of the field $k$, $a,b\in k$, $a\neq \pm b$.

Now we must prove that for all $\varphi \in \mathrm{Aut}k$ and all $a,b\in k$
such that $a\neq \pm b$ the system of words (\ref{words_list}) where $%
w_{0}=0 $, $w_{\lambda }=\varphi \left( \lambda \right) x_{1}$ for all $%
\lambda \in k $, $w_{-}=-x_{1}$, $w_{+}=x_{1}+x_{2}$, $w_{\cdot
}=ax_{1}x_{2}+bx_{2}x_{1}$ fulfills condition Op2. It means that we must
build for every $F=F\left( X\right) \in \mathrm{Ob}\Theta ^{0}$ an
isomorphism $\sigma _{F}:F\rightarrow F_{W}^{\ast }$ such that $\sigma
_{F}\mid _{X}=id_{X}$.

We will prove, first of all, that $H_{W}^{\ast }\in \Theta $ for every $H\in
\Theta $. Operations defined by $w_{0}$, $w_{-}$, $w_{+}$ coincide with $0$, 
$-$, $+$. So identities of the variety $\Theta $ (axioms of the linear
algebra) relating to these operations fulfill in $H_{W}^{\ast }$. Hence we
only need to check the axioms that involve the operations defined by $%
w_{\cdot }$ and $w_{\lambda }$ ($\lambda \in k$). As above we denote these
operations by $\times $ and by $\lambda \ast $.%
\[
\lambda \ast \left( x+y\right) =\varphi \left( \lambda \right) \left(
x+y\right) =\varphi \left( \lambda \right) x+\varphi \left( \lambda \right)
y=\lambda \ast x+\lambda \ast y, 
\]%
\[
\left( \lambda \mu \right) \ast x=\varphi \left( \lambda \mu \right)
x=\varphi \left( \lambda \right) \varphi \left( \mu \right) x=\varphi \left(
\lambda \right) \left( \mu \ast x\right) =\lambda \ast \left( \mu \ast
x\right) , 
\]%
\[
\left( \lambda +\mu \right) \ast x=\varphi \left( \lambda +\mu \right)
x=\left( \varphi \left( \lambda \right) +\varphi \left( \mu \right) \right)
x=\varphi \left( \lambda \right) x+\varphi \left( \mu \right) x=\lambda \ast
x+\mu \ast x, 
\]%
\[
1\ast x=\varphi \left( 1\right) x=1x=x, 
\]%
\[
x\times \left( y+z\right) =ax\left( y+z\right) +b\left( y+z\right)
x=axy+axz+byx+bzx=x\times y+x\times z, 
\]%
\[
\left( y+z\right) \times x=a\left( y+z\right) x+bx\left( y+z\right)
=ayx+azx+bxy+bxz=y\times x+z\times x, 
\]%
\[
\lambda \ast \left( x\times y\right) =\varphi \left( \lambda \right) \left(
axy+byx\right) =a\left( \varphi \left( \lambda \right) x\right) y+by\left(
\varphi \left( \lambda \right) x\right) =\left( \varphi \left( \lambda
\right) x\right) \times y= 
\]%
\[
\left( \lambda \ast x\right) \times y=x\times \left( \lambda \ast y\right) 
\]%
fulfills for every $x,y,z\in H$, $\lambda ,\mu \in k$.

Hence there exists a homomorphism $\sigma _{F}:F\rightarrow F_{W}^{\ast }$
such that $\sigma _{F}\mid _{X}=id_{X}$ for every $F=F\left( X\right) \in 
\mathrm{Ob}\Theta ^{0}$. Our goal is to prove that these homomorphisms are
isomorphisms. We will prove by induction by $i$ that 
\begin{equation}
\sigma _{F}\left( F_{i}\right) =F_{i}.  \label{epi_homo}
\end{equation}%
for every $i\in 
%TCIMACRO{\U{2115} }%
%BeginExpansion
\mathbb{N}
%EndExpansion
$. If $X=\left\{ x_{1},\ldots ,x_{n}\right\} $ then every element of $F_{1}$
has form $\lambda _{1}x_{1}+\ldots +\lambda _{n}x_{n}$, where $\lambda
_{1},\ldots ,\lambda _{n}\in k$.%
\[
\sigma _{F}\left( \lambda _{1}x_{1}+\ldots +\lambda _{n}x_{n}\right)
=\lambda _{1}\ast \sigma _{F}\left( x_{1}\right) +\ldots +\lambda _{n}\ast
\sigma _{F}\left( x_{n}\right) =\varphi \left( \lambda _{1}\right)
x_{1}+\ldots +\varphi \left( \lambda _{n}\right) x_{n}, 
\]%
so $\sigma _{F}\left( F_{1}\right) \subset F_{1}$.%
\[
\sigma _{F}\left( \varphi ^{-1}\left( \lambda _{1}\right) x_{1}+\ldots
+\varphi ^{-1}\left( \lambda _{n}\right) x_{n}\right) =\lambda
_{1}x_{1}+\ldots +\lambda _{n}x_{n}, 
\]%
so $\sigma _{F}\left( F_{1}\right) =F_{1}$.

Let (\ref{epi_homo}) proved for $i$ such that $1\leq i<r$. Every element of $%
F_{r}$ is a linear combination of the monomials of the form $uv$, where $%
u\in F_{i}$, $v\in F_{j}$, $i+j=r$.%
\[
\sigma _{F}\left( uv\right) =\sigma _{F}\left( u\right) \times \sigma
_{F}\left( v\right) =a\sigma _{F}\left( u\right) \sigma _{F}\left( v\right)
+b\sigma _{F}\left( v\right) \sigma _{F}\left( u\right) ,
\]%
so $\sigma _{F}\left( F_{r}\right) \subset F_{r}$, because, by our
assumption, $\sigma _{F}\left( u\right) \in F_{i}$, $\sigma _{F}\left(
v\right) \in F_{j}$. Also, if $u=\sigma _{F}\left( \widetilde{u}\right) $, $%
v=\sigma _{F}\left( \widetilde{v}\right) $, where $\widetilde{u}\in F_{r}$, $%
\widetilde{v}\in F_{t}$, then%
\[
\sigma _{F}\left( \widetilde{u}\widetilde{v}\right) =\sigma _{F}\left( 
\widetilde{u}\right) \times \sigma _{F}\left( \widetilde{v}\right) =u\times
v=auv+bvu
\]%
\[
\sigma _{F}\left( \widetilde{v}\widetilde{u}\right) =\sigma _{F}\left( 
\widetilde{v}\right) \times \sigma _{F}\left( \widetilde{u}\right) =v\times
u=avu+buv=buv+avu,
\]%
fulfills. $a\neq \pm b$, so the matrix $\left( 
\begin{array}{cc}
a & b \\ 
b & a%
\end{array}%
\right) $ is regular, hence there exist $\alpha ,\beta \in k$ such that 
\[
uv=\alpha \sigma _{F}\left( \widetilde{u}\widetilde{v}\right) +\beta \sigma
_{F}\left( \widetilde{v}\widetilde{u}\right) =\sigma _{F}\left( \varphi
^{-1}\left( \alpha \right) \widetilde{u}\widetilde{v}+\varphi ^{-1}\left(
\beta \right) \widetilde{v}\widetilde{u}\right) .
\]%
Therefore $\sigma _{F}\left( F_{r}\right) =F_{r}$. We can conclude that $%
\sigma _{F}$ is an epimorphism.

Now we will prove that $\ker \sigma _{F}=0$. Let $f\in \ker \sigma
_{F}\subset F\left( X\right) $. There exists $m\in 
%TCIMACRO{\U{2115} }%
%BeginExpansion
\mathbb{N}
%EndExpansion
$ such that $f\in \bigoplus\limits_{i=1}^{m}F_{i}$. $\sigma _{F}\left(
\bigoplus\limits_{i=1}^{m}F_{i}\right) =\bigoplus\limits_{i=1}^{m}F_{i}$ by (%
\ref{epi_homo}). $\sigma _{F}$ is a linear mapping from the linear space $%
\bigoplus\limits_{i=1}^{m}F_{i}$ with the original multiplication by scalars
in $F$ to the $\left( \bigoplus\limits_{i=1}^{m}F_{i}\right) _{W}^{\ast }$ -
the linear space $\bigoplus\limits_{i=1}^{m}F_{i}$ with the multiplication
by scalars which we denote by $\ast $. From formulas $\sum\limits_{i=1}^{k}%
\left( \lambda _{i}\ast e_{i}\right) =\sum\limits_{i=1}^{k}\varphi \left(
\lambda _{i}\right) e_{i}$ and $\sum\limits_{i=1}^{k}\lambda
_{i}e_{i}=\sum\limits_{i=1}^{k}\left( \varphi ^{-1}\left( \lambda
_{i}\right) \ast e_{i}\right) $ we can conclude that if $E$ is a basis of
the linear space $\bigoplus\limits_{i=1}^{m}F_{i}$ then $E$ is a basis of
the linear space $\left( \bigoplus\limits_{i=1}^{m}F_{i}\right) _{W}^{\ast }$%
. So $\dim \bigoplus\limits_{i=1}^{m}F_{i}=\dim \left(
\bigoplus\limits_{i=1}^{m}F_{i}\right) _{W}^{\ast }<\infty $, therefore $%
\ker \left( \sigma _{F}\mid \bigoplus\limits_{i=1}^{m}F_{i}\right) =0$ and $%
f=0$.
\end{proof}

\section{Group $\mathfrak{A/Y}$.\label{group}}

\setcounter{equation}{0}

From now on, $W$ is a system of words (\ref{words_list}) which fulfills
conditions Op1 and Op2.

The decomposition (\ref{decomp}) is not split in general case, i. e. $%
\mathfrak{S\cap Y\neq }\left\{ 1\right\} $ in general case. The strongly
stable automorphism $\Phi $ of the category $\Theta ^{0}$ which corresponds
to the system of words $W$ is inner, by \cite[Lemma 3]{PlotkinZhitAutCat},
if and only if for every $F\in \mathrm{Ob}\Theta ^{0}$ there exists an
isomorphism $c_{F}:F\rightarrow F_{W}^{\ast }$ such that $c_{F}\alpha
=\alpha c_{D}$ fulfills for every $\left( \alpha :D\rightarrow F\right) \in 
\mathrm{Mor}\Theta ^{0}$ (by \cite[Remark 3.1]{TsurkovAutomEquiv} $\alpha $
is also a homomorphism from $D_{W}^{\ast }$ to $F_{W}^{\ast }$).

Hear we need to prove one technical lemma.

\begin{lemma}
\label{F/F^2}If $F=F\left( X\right) \in \mathrm{Ob}\Theta ^{0}$ and $%
c_{F}:F\rightarrow F_{W}^{\ast }$ is an isomorphism then there exists an
isomorphism $c_{i}:F/F^{i}\rightarrow F_{W}^{\ast }/F^{i}$ such that $\chi
_{i}^{\ast }c_{F}=c_{i}\chi _{i}$, where $\chi _{i}:F\rightarrow F/F^{i}$
and $\chi _{i}^{\ast }:F_{W}^{\ast }\rightarrow F_{W}^{\ast }/F^{i}$ are
natural homomorphisms, $i\in 
%TCIMACRO{\U{2115} }%
%BeginExpansion
\mathbb{N}
%EndExpansion
$.
\end{lemma}

\begin{proof}
If $H\in \Theta $ and $I$ is an ideal of $H$. If $\lambda \in k$, $y\in I$, $%
h\in H$, then $\lambda \ast y=\varphi \left( \lambda \right) y\in I$, $%
y\times h=ayh+bhy\in I$, analogously $h\times y\in I$. Therefore $I$ is an
ideal of $H_{W}^{\ast }$. Hence $F^{i}$ is an ideal of $F_{W}^{\ast }$.

If $\sigma _{F}:F\rightarrow F_{W}^{\ast }$ is an isomorphism such that $%
\sigma _{F}\mid _{X}=id_{X}$, then by (\ref{epi_homo}) we have $%
c_{F}^{-1}\left( F^{i}\right) =c_{F}^{-1}\sigma _{F}\left( F^{i}\right)
=F^{i}$ because $c_{F}^{-1}\sigma _{F}:F\rightarrow F$ is an isomorphism. So 
$c_{F}\left( F^{i}\right) =F^{i}$. It finishes the proof.
\end{proof}

\begin{proposition}
\label{stable_inner}The strongly stable automorphism $\Phi $ which
corresponds to the system of words $W$ is inner if and only if $\varphi
=id_{k}$ and $b=0$.
\end{proposition}

\begin{proof}
We suppose that strongly stable automorphism $\Phi $ which corresponds to
the system of words $W$ is inner. We assume that $\varphi \neq id_{k}$, i.,
e., there exists $\lambda \in k$ such that $\varphi \left( \lambda \right)
\neq \lambda $. We denote $F=F\left( x\right) $. We take $\alpha \in \mathrm{%
End}F$, such that $\alpha \left( x\right) =\lambda x$. We suppose that $%
c_{F}:F\rightarrow F_{W}^{\ast }$ is an isomorphism. $c_{2}$ is defined as
in the Lemma \ref{F/F^2}, and we by this Lemma we have:%
\[
\chi _{2}^{\ast }c_{F}\left( x\right) =c_{2}\chi _{2}\left( x\right) =\mu
\ast \chi _{2}^{\ast }\left( x\right) =\chi _{2}^{\ast }\left( \mu \ast
x\right) =\chi _{2}^{\ast }\left( \varphi \left( \mu \right) x\right) ,
\]%
where operations in algebra $F_{W}^{\ast }/F^{2}$ we denote by same symbols
as operations in algebra $F_{W}^{\ast }$ and $\mu \in k\setminus \left\{
0\right\} $. Therefore $c_{F}\left( x\right) \equiv \varphi \left( \mu
\right) x\left( \func{mod}F^{2}\right) $. $\alpha \left( F^{2}\right)
\subset F^{2}$ fulfils, so%
\[
\alpha c_{F}\left( x\right) =\alpha \left( \varphi \left( \mu \right)
x+f_{2}\right) \equiv \alpha \left( \varphi \left( \mu \right) x\right)
=\varphi \left( \mu \right) \alpha \left( x\right) =\varphi \left( \mu
\right) \lambda x\left( \func{mod}F^{2}\right) ,
\]%
where $f_{2}\in F^{2}$. 
\[
c_{F}\alpha \left( x\right) =c_{F}\left( \lambda x\right) =\lambda \ast
c_{F}\left( x\right) =\varphi \left( \lambda \right) c_{F}\left( x\right)
\equiv \varphi \left( \lambda \right) \varphi \left( \mu \right) x\left( 
\func{mod}F^{2}\right) .
\]%
$\mu \neq 0$, so $\varphi \left( \mu \right) \neq 0$, $\varphi \left(
\lambda \right) \neq \lambda $ hence $\alpha c_{F}\neq c_{F}\alpha $. This
contradiction proves that $\varphi =id_{k}$.

Now we denote $F=F\left( x_{1},x_{2}\right) \in \mathrm{Ob}\Theta ^{0}$. By
our assumption there exists an isomorphism $c_{F}:F\rightarrow F_{W}^{\ast }$
such that $c_{F}\alpha =\alpha c_{F}$ fulfills for every $\alpha \in \mathrm{%
End}F$. $c_{2}$ is defined as in the Lemma \ref{F/F^2}. $\alpha \left(
F^{2}\right) \subset F^{2}$ so we can define the homomorphism $\widetilde{%
\alpha }:F/F^{2}\rightarrow F/F^{2}$ such that $\widetilde{\alpha }\chi
_{2}=\chi _{2}\alpha $. From $c_{F}\alpha =\alpha c_{F}$ we can conclude $%
c_{2}\widetilde{\alpha }=\widetilde{\alpha }c_{2}$ fulfills. By Lemma \ref%
{F/F^2} $c_{2}$ is a regular linear mapping. We can take the endomorphisms $%
\alpha $ such that $\widetilde{\alpha }$ will be an arbitrary linear mapping
from $k^{2}$ to $k^{2}$. Therefore $c_{2}$ must be a regular linear mapping
from $k^{2}$ to $k^{2}$ which commutate with all linear mappings from $k^{2}$
to $k^{2}$. Hence $c_{2}$ must be a scalar mapping, i.e., 
\[
\chi _{2}^{\ast }c_{F}\left( x_{i}\right) =c_{2}\chi _{2}\left( x_{i}\right)
=\lambda \chi _{2}^{\ast }\left( x_{i}\right) =\chi _{2}^{\ast }\left(
\lambda x_{i}\right) , 
\]%
where $\lambda \in k\setminus \left\{ 0\right\} $, $i=1,2$. Therefore $%
c_{F}\left( x_{i}\right) =\lambda x_{i}+f_{i}$, where $f_{i}\in F^{2}$, $%
i=1,2$. We can remark that now we consider the case when $\varphi =id_{k}$,
hence we need not distinguish between multiplication by scalar in $F$ and $%
F_{W}^{\ast }$.

Now we take $\alpha \in \mathrm{End}F$ such that $\alpha \left( x_{1}\right)
=x_{1}x_{2}$, $\alpha \left( x_{2}\right) =0$. If $u$ is a monomial which
contain only entries of $x_{1}$, then $\deg _{x_{1}}\alpha \left( u\right)
+\deg _{x_{2}}\alpha \left( u\right) =2\deg _{x_{1}}u$. If a monomial $u$
contain at least one entry of $x_{2}$, then $\alpha \left( u\right) =0$.
Hence $\alpha \left( F^{2}\right) \subset F^{3}$. So we have 
\[
c_{F}\alpha \left( x_{1}\right) =c_{F}\left( x_{1}x_{2}\right) =c_{F}\left(
x_{1}\right) \times c_{F}\left( x_{2}\right) =ac_{F}\left( x_{1}\right)
c_{F}\left( x_{2}\right) +bc_{F}\left( x_{2}\right) c_{F}\left( x_{1}\right)
= 
\]%
\[
a\left( \lambda x_{1}+f_{1}\right) \left( \lambda x_{2}+f_{2}\right)
+b\left( \lambda x_{2}+f_{2}\right) \left( \lambda x_{1}+f_{1}\right) \equiv
a\lambda ^{2}x_{1}x_{2}+b\lambda ^{2}x_{2}x_{1}\left( \func{mod}F^{3}\right)
. 
\]%
\[
\alpha c_{F}\left( x_{1}\right) =\alpha \left( \lambda x_{1}+f_{1}\right)
\equiv \lambda x_{1}x_{2}\left( \func{mod}F^{3}\right) . 
\]%
Hence we conclude $b=0$ from $c_{F}\alpha =\alpha c_{F}$.

If $b=0$, i. e., $w_{\cdot }=ax_{1}x_{2}$, $a\neq 0$, then we take $%
c_{F}\left( f\right) =a^{-1}f$ for every $F\in \mathrm{Ob}\Theta ^{0}$ and
every $f\in F$. It is obvious that $c_{F}$ is a regular linear mapping. 
\[
c_{F}\left( f_{1}\right) \times c_{F}\left( f_{2}\right) =ac_{F}\left(
f_{1}\right) c_{F}\left( f_{2}\right) =a\left( a^{-1}f_{1}\right) \left(
a^{-1}f_{2}\right) =a^{-1}f_{1}f_{2}=c_{F}\left( f_{1}f_{2}\right) . 
\]%
for every $f_{1},f_{2}\in F$. So $c_{F}:F\rightarrow F_{W}^{\ast }$ is an
isomorphism. It fulfils%
\[
c_{F}\alpha \left( d\right) =a^{-1}\alpha \left( d\right) =\alpha \left(
a^{-1}d\right) =\alpha c_{F}\left( d\right) 
\]%
for every $\left( \alpha :D\rightarrow F\right) \in \mathrm{Mor}\Theta ^{0}$
and every $d\in D$.
\end{proof}

\begin{proposition}
\label{str_stb_group}The group $\mathfrak{S\cong }G\mathfrak{\leftthreetimes 
}\mathrm{Aut}k$, where $G$ is the group of all regular $2\times 2$ matrices
over field $k$, which have a form $\left( 
\begin{array}{cc}
a & b \\ 
b & a%
\end{array}%
\right) $ and every $\varphi \in \mathrm{Aut}k$ acts on the group $G$ by
this way: $\varphi \left( 
\begin{array}{cc}
a & b \\ 
b & a%
\end{array}%
\right) =\left( 
\begin{array}{cc}
\varphi \left( a\right) & \varphi \left( b\right) \\ 
\varphi \left( b\right) & \varphi \left( a\right)%
\end{array}%
\right) $.
\end{proposition}

\begin{proof}
We will define the mapping $\tau :G\mathfrak{\leftthreetimes }\mathrm{Aut}%
k\rightarrow \mathfrak{S}$. If $g\varphi \in G\mathfrak{\leftthreetimes }%
\mathrm{Aut}k$, where $g=\left( 
\begin{array}{cc}
a & b \\ 
b & a%
\end{array}%
\right) $, then we define $\tau \left( g\varphi \right) =\Phi \in \mathfrak{S%
}$, where $\Phi $ corresponds to the system of words $W$ with $w_{\lambda
}=\varphi \left( \lambda \right) x_{1}$ for every $\lambda \in k$ and $%
w_{\cdot }=ax_{1}x_{2}+bx_{2}x_{1}$. By Section \ref{operations} and Theorem %
\ref{words} $\tau $ is bijection.

We consider $\tau \left( g_{1}\varphi _{1}\right) =\Phi _{1}$ and $\tau
\left( g_{2}\varphi _{2}\right) =\Phi _{2}$, where $g_{1}\varphi
_{1},g_{2}\varphi _{2}\in G\mathfrak{\leftthreetimes }\mathrm{Aut}k$ and $%
g_{1}=\left( 
\begin{array}{cc}
a_{1} & b_{1} \\ 
b_{1} & a_{1}%
\end{array}%
\right) $, $g_{2}=\left( 
\begin{array}{cc}
a_{2} & b_{2} \\ 
b_{2} & a_{2}%
\end{array}%
\right) $. Both these strongly stable automorphisms preserves all objects of 
$\Theta ^{0}$ and acts on morphisms of $\Theta ^{0}$ by theirs systems of
bijections $\left\{ s_{F}^{\Phi _{i}}:F\rightarrow F\mid F\in \mathrm{Ob}%
\Theta ^{0}\right\} $, for $i=1,2$, according the formula (\ref%
{biject_action}). We have $\Phi _{2}\Phi _{1}\left( \alpha \right)
=s_{F}^{\Phi _{2}}s_{F}^{\Phi _{1}}\alpha \left( s_{D}^{\Phi _{1}}\right)
^{-1}\left( s_{D}^{\Phi _{2}}\right) ^{-1}$ for every $\left( \alpha
:D\rightarrow F\right) \in \mathrm{Mor}\Theta ^{0}$. So strongly stable
automorphism $\Phi _{2}\Phi _{1}=\tau \left( g_{2}\varphi _{2}\right) \tau
\left( g_{1}\varphi _{1}\right) $ preserves all objects of $\Theta ^{0}$ and
acts on morphisms of $\Theta ^{0}$ by system of bijections%
\[
\left\{ s_{F}^{\Phi _{2}}s_{F}^{\Phi _{1}}:F\rightarrow F\mid F\in \mathrm{Ob%
}\Theta ^{0}\right\} . 
\]%
This system of bijections satisfies the conditions B1 and B2, so we can
define the words $w_{\lambda }^{\Phi _{2}\Phi _{1}}$ for every $\lambda \in
k $ and $w_{\cdot }^{\Phi _{2}\Phi _{1}}$ which correspond to the
automorphism $\Phi _{2}\Phi _{1}$ by formula (\ref{der_veb_opr}). The words $%
w_{\lambda }^{\Phi _{i}}(\lambda \in k)$ and $w_{\cdot }^{\Phi _{i}}$ which
correspond to the automorphism $\Phi _{i}$ have forms $w_{\lambda }^{\Phi
_{i}}=\varphi _{i}\left( \lambda \right) x_{1}(\lambda \in k)$ and $w_{\cdot
}^{\Phi _{i}}=a_{i}x_{1}x_{2}+b_{i}x_{2}x_{1}$ for $i=1,2$. So 
\[
w_{\lambda }^{\Phi _{2}\Phi _{1}}=s_{F}^{\Phi _{2}}s_{F}^{\Phi _{1}}\left(
\lambda x_{1}\right) =s_{F}^{\Phi _{2}}\left( w_{\lambda }^{\Phi
_{1}}\right) =s_{F}^{\Phi _{2}}\left( \varphi _{1}\left( \lambda \right)
x_{1}\right) =\varphi _{2}\left( \varphi _{1}\left( \lambda \right) \right)
x_{1}=\left( \varphi _{2}\varphi _{1}\right) \left( \lambda \right) x_{1} 
\]%
for every $\lambda \in k$ and%
\[
w_{\cdot }^{\Phi _{2}\Phi _{1}}=s_{F}^{\Phi _{2}}s_{F}^{\Phi _{1}}\left(
x_{1}x_{2}\right) =s_{F}^{\Phi _{2}}\left( w_{\cdot }^{\Phi _{1}}\right)
=s_{F}^{\Phi _{2}}\left( a_{1}x_{1}x_{2}+b_{1}x_{2}x_{1}\right) = 
\]%
\[
\varphi _{2}\left( a_{1}\right) s_{F}^{\Phi _{2}}\left( x_{1}x_{2}\right)
+\varphi _{2}\left( b_{1}\right) s_{F}^{\Phi _{2}}\left( x_{2}x_{1}\right) = 
\]%
\[
\varphi _{2}\left( a_{1}\right) \left(
a_{2}x_{1}x_{2}+b_{2}x_{2}x_{1}\right) +\varphi _{2}\left( b_{1}\right)
\left( a_{2}x_{2}x_{1}+b_{2}x_{1}x_{2}\right) = 
\]%
\[
\left( \varphi _{2}\left( a_{1}\right) a_{2}+\varphi _{2}\left( b_{1}\right)
b_{2}\right) x_{1}x_{2}+\left( \varphi _{2}\left( a_{1}\right) b_{2}+\varphi
_{2}\left( b_{1}\right) a_{2}\right) x_{2}x_{1}. 
\]%
because $s_{F}^{\Phi _{i}}:F\rightarrow F_{W_{i}}^{\ast }$ is an
isomorphism, $i=1,2$. Hence%
\[
\Phi _{2}\Phi _{1}=\tau \left( g_{2}\varphi _{2}\right) \tau \left(
g_{1}\varphi _{1}\right) =\tau \left( g_{2}\varphi _{2}\left( g_{1}\right)
\varphi _{2}\varphi _{1}\right) =\tau \left( g_{2}\varphi _{2}\cdot
g_{1}\varphi _{1}\right) . 
\]
\end{proof}

\begin{corollary}
\label{intersection}Group $\mathfrak{S\cap Y}$ is isomorphic to the group $%
k^{\ast }I_{2}$ of the regular $2\times 2$ scalar matrices over field $k$.
\end{corollary}

\begin{proof}
By Propositions \ref{stable_inner} and\ \ref{str_stb_group}.
\end{proof}

\begin{corollary}
$\mathfrak{A/Y\cong }\left( G\mathfrak{/}k^{\ast }I_{2}\right) \mathfrak{%
\leftthreetimes }\mathrm{Aut}k$.
\end{corollary}

\begin{proof}
By Proposition \ref{str_stb_group} and Corollary 1 we have that $\mathfrak{%
A/Y\cong }\left( G\mathfrak{\leftthreetimes }\mathrm{Aut}k\right) \mathfrak{/%
}k^{\ast }I_{2}$. And we have $\left( G\mathfrak{\leftthreetimes }\mathrm{Aut%
}k\right) \mathfrak{/}k^{\ast }I_{2}\mathfrak{\cong }\left( G\mathfrak{/}%
k^{\ast }I_{2}\right) \mathfrak{\leftthreetimes }\mathrm{Aut}k$ because $%
k^{\ast }I_{2}\vartriangleleft G$ and for every $\varphi \in \mathrm{Aut}k$ $%
\varphi \left( k^{\ast }I_{2}\right) \subset k^{\ast }I_{2}$ fulfills.
\end{proof}

The symmetric group of the set which has $2$ elements - $\mathbf{S}_{\mathbf{%
2}}$ can be embedded in the multiplicative structure of the algebra $\mathbf{%
M}_{\mathbf{2}}\left( k\right) $ of the $2\times 2$ matrices over field $k$: 
$\mathbf{S}_{\mathbf{2}}\ni \left( 12\right) \rightarrow \left( 
\begin{array}{cc}
0 & 1 \\ 
1 & 0%
\end{array}%
\right) \in \mathbf{M}_{\mathbf{2}}\left( k\right) $, so $G\cong U\left( k%
\mathbf{S}_{\mathbf{2}}\right) $, where $U\left( k\mathbf{S}_{\mathbf{2}%
}\right) $ is the group of all invertible elements of the group algebra $k%
\mathbf{S}_{\mathbf{2}}$. Also $k^{\ast }I_{2}\cong U\left( k\left\{
e\right\} \right) $, where $e\in \mathbf{S}_{\mathbf{2}}$, $k\left\{
e\right\} $ is a subalgebra of $k\mathbf{S}_{\mathbf{2}}$, $U\left( k\left\{
e\right\} \right) $ is a group of all invertible elements of this
subalgebra. Therefore $\mathfrak{A/Y\cong }\left( U\left( k\mathbf{S}_{%
\mathbf{2}}\right) \mathfrak{/}U\left( k\left\{ e\right\} \right) \right) 
\mathfrak{\leftthreetimes }\mathrm{Aut}k$, where every $\varphi \in \mathrm{%
Aut}k$ acts on the algebra $k\mathbf{S}_{\mathbf{2}}$ by natural way: $%
\varphi \left( ae+b\left( 12\right) \right) =\varphi \left( a\right)
e+\varphi \left( b\right) \left( 12\right) $.

\section{Example of two linear algebras which are automorphically equivalent
but not\textbf{\ }geometrically equivalent.}

\setcounter{equation}{0}

We take $k=%
%TCIMACRO{\U{211a} }%
%BeginExpansion
\mathbb{Q}
%EndExpansion
$. $\Theta $ will be the variety of all linear algebras over $k$. $H$ will
be the $2$-generated linear algebra, which is free in the variety
corresponding to the identity $\left( x_{1}x_{1}\right) x_{2}=0$. We
consider the strongly stable automorphism $\Phi $ of the category $\Theta
^{0}$ corresponding to the system of words $W$, where $b\neq 0$. Algebras $H$
and $H_{W}^{\ast }$ are automorphically equivalent by \cite[Theorem 5.1]%
{TsurkovAutomEquiv}.

\begin{proposition}
\label{not_geom_equiv}Algebras $H$ and $H_{W}^{\ast }$ are not geometrically
equivalent.
\end{proposition}

\begin{proof}
Let $F=F\left( x_{1},x_{2}\right) $. The ideal $I=Id\left( H,\left\{
x_{1},x_{2}\right\} \right) $ of the all two-variables identities which are
fulfill in the algebra $H$ will be the smallest $H$-closed set in $F$,
because $I=\left( 0\right) _{H}^{\prime \prime }$, where $0\in F$. If
algebras $H$ and $H_{W}^{\ast }$ are geometrically equivalent then the
structures of the $H$-closed sets and of the $H_{W}^{\ast }$-closed sets in $%
F$ coincide. Hence $I$ must be the smallest $H_{W}^{\ast }$-closed set in $F$%
.

By \cite[Remark 5.1]{TsurkovAutomEquiv} 
\begin{equation}
T\rightarrow \sigma _{F}T  \label{closed_bijection}
\end{equation}%
is a bijection from the structure of the $H_{W}^{\ast }$-closed sets in $F$
to the structure of the $H$-closed sets in $F$. Hear $\sigma
_{F}:F\rightarrow F_{W}^{\ast }$ is an isomorphism from condition Op2. It is
clear that the bijection (\ref{closed_bijection}) preserves inclusions of
sets. So it transforms the smallest $H_{W}^{\ast }$-closed set to the
smallest $H$-closed set, i. e. $I=\sigma _{F}I$ must fulfills.

It is obviously that $I\subset F^{3}$. By (\ref{epi_homo}) $\sigma
_{F}I\subset F^{3}$. We will compare the linear subspaces $I/F^{4}$ and $%
\left( \sigma _{F}I\right) /F^{4}$. $I=\left\langle \alpha \left( \left(
x_{1}x_{1}\right) x_{2}\right) \mid \alpha \in \mathrm{End}F\right\rangle $.
Let $\alpha \left( x_{i}\right) \equiv \alpha _{1i}x_{1}+\alpha
_{2i}x_{2}\left( \func{mod}F^{2}\right) $, where $i=1,2$, $\alpha _{ji}\in k$%
. Then%
\[
\alpha \left( \left( x_{1}x_{1}\right) x_{2}\right) \equiv \left( \left(
\alpha _{11}x_{1}+\alpha _{21}x_{2}\right) \left( \alpha _{11}x_{1}+\alpha
_{21}x_{2}\right) \right) \left( \alpha _{12}x_{1}+\alpha _{22}x_{2}\right)
\left( \func{mod}F^{4}\right) .
\]%
We achieve after the extending of brackets that $I/F^{4}$ is a subspace of
the linear space spanned by the elements of $F^{3}/F^{4}$ which have form $%
\left( x_{i}x_{j}\right) x_{k}+F^{4}$, where $i,j,k=1,2$. But%
\[
\sigma _{F}I\ni \sigma _{F}\left( \left( x_{1}x_{1}\right) x_{2}\right)
=a\sigma _{F}\left( x_{1}x_{1}\right) \sigma _{F}\left( x_{2}\right)
+b\sigma _{F}\left( x_{2}\right) \sigma _{F}\left( x_{1}x_{1}\right) =
\]%
\[
=a\left( a+b\right) \left( x_{1}x_{1}\right) x_{2}+b\left( a+b\right)
x_{2}\left( x_{1}x_{1}\right) .
\]%
We have that $a+b\neq 0$, $b\neq 0$, so $I/F^{4}\neq \left( \sigma
_{F}I\right) /F^{4}$ and $I\neq \sigma _{F}I$. This contradiction proves
that algebras $H$ and $H_{W}^{\ast }$ are not geometrically equivalent.
\end{proof}

\end{document}